\documentclass[11pt,oneside, a4paper]{amsart}

\usepackage[english]{babel} 
\usepackage[T1]{fontenc} 
\usepackage{amsmath}
\usepackage{amssymb} 
\usepackage{float}
\usepackage{amsfonts}
\usepackage{mathtools}
\usepackage[utf8]{inputenc}
\usepackage[mathcal]{eucal} 
\usepackage{enumerate}
\usepackage{amsthm} 
\usepackage{hyperref}
\usepackage[totalwidth=14cm,totalheight=20cm]{geometry}

\newtheorem{defi}{Definition}[section] 
\newtheorem{teo}[defi]{Theorem}
\newtheorem{cor}[defi]{Corollary} 
\newtheorem{lemma}[defi]{Lemma}
\newtheorem{prop}[defi]{Proposition}
\newtheorem{oss}[defi]{Remark}

\newcommand{\Pp}{\mathbb{P}}
\newcommand{\R}{\mathbb{R}}

\newcommand{\K}{\kappa}

\newcommand{\dPSL}{\mathbb{P}SL(2,\mathbb{R})\times \mathbb{P}SL(2,\mathbb{R})}
\newcommand{\Sym}{\text{Sym}}
\newcommand{\PSL}{\mathbb{P}SL}
\newcommand{\T}{\mathrm{Teich}}
\newcommand{\dT}{\T(S) \times \T(S)}
\newcommand{\Ima}{\mathrm{Im}}
\newcommand{\Isom}{\mathrm{Isom}}
\newcommand{\codim}{\mathrm{codim}}

\DeclareMathAlphabet{\mathpzc}{OT1}{pzc}{m}{it}

\title[Prescribing metrics on the boundary of Anti-de Sitter 3-manifolds]{Prescribing metrics on the boundary of \\ Anti-de Sitter 3-manifolds} 
\author{Andrea Tamburelli}
\date{\today}
\thanks{}

\begin{document}

\begin{abstract} 
We prove that given two metrics $g_{+}$ and $g_{-}$ with curvature $\kappa <-1$ on a closed, oriented surface $S$ of genus $\tau\geq 2$, there exists an $AdS_{3}$ manifold $N$ with smooth, space-like, strictly convex boundary such that the induced metrics on the two connected components of $\partial N$ are equal to $g_{+}$ and $g_{-}$. Using the duality between convex space-like surfaces in $AdS_{3}$, we obtain an equivalent result about the prescription of the third fundamental form. This answers partially Question 3.5 in \cite{domande}. 
\end{abstract}

\maketitle

\setcounter{tocdepth}{1}
\tableofcontents

\section*{Introduction}\label{Intro}
The $3$-dimensional Anti-de Sitter space $AdS_{3}$ is the Lorentzian analogue of hyperbolic space, i.e. it is the local model of Lorentzian $3$-manifolds with constant sectional curvature $-1$. An $AdS_{3}$ spacetime is an oriented and time-oriented manifold locally modelled on $AdS_{3}$. A particular class of Anti-de Sitter $3$-manifold, called globally hyperbolic maximal compact (GHMC), has attracted much attention since the pioneering work of Mess (\cite{Mess}), who pointed out many connections to Teichm\"uller theory and many similarities to hyperbolic quasi-Fuchsian manifolds. \\
\indent A GHMC $AdS_{3}$ spacetime $M$ is topologically a product $S\times \R$, where $S$ is a closed, oriented surface of genus $\tau\ge 2$, diffeomorphic to a Cauchy surface embedded in $M$. The holonomy representation of the fundamental group of $S$ into the isometry group of orientation and time-orientation preserving isometries of $AdS_{3}$, which can be identified with $\dPSL$, provides a bijection between the space of GHMC $AdS_{3}$ structures on $M$ and the product of two copies of the Teichm\"uller space of $S$. Moreover, $M$ contains a convex core $C(M)$, whose boundary (when $C(M)$ is not a totally geodesic $2$-manifold) consists of the disjoint union of two hyperbolic surfaces pleated along a geodesic lamination.\\
\indent As a consequence, it is possible to formulate many classical questions of quasi-Fuchsian manifolds even in this Lorentzian setting. For example one can ask if it is possible to prescribe the induced metrics on the boundary of the convex core and the answer is very similar in both settings, where the existence of a quasi-Fuchsian manifold (as a consequence of results in \cite{EM} and \cite{Lab}) and a GHMC $AdS_{3}$ manifold (\cite{Diallo}) with a prescribed metric on the boundary of the convex core has been proved, but uniqueness is still unknown. \\
\indent Another interesting question deals with the prescription of the metrics on the boundary of a larger compact, convex subset $K$ with two smooth, strictly convex, space-like boundary components in a GHMC $AdS_{3}$ manifold. By the Gauss formula, the boundaries have curvature $\kappa <-1$. We can ask if it is possible to realize every couple of metrics, satisfying the condition on the curvature, on a surface $S$ via this construction. The analogous question has a positive answer in a hyperbolic setting (\cite{Lab}), where even a uniqueness result holds (\cite{SConvex}). In this paper, we will follow a construction inspired by the work of Labourie (\cite{Lab}), in order to obtain a positive answer in the Anti-de Sitter world. The main result of the paper is thus the following:\\
\\
\indent {\bf Corollary 3.3}  {\it For every couple of metrics $g_{+}$ and $g_{-}$ on $S$ with curvature less than $-1$, there exists a globally hyperbolic convex compact $AdS_{3}$ manifold $K\cong S\times [0,1]$, whose induced metrics on the boundary are exactly $g_{\pm}$.}\\
\\
\indent Using the duality between space-like surfaces in Anti-de Sitter space, we obtain an analogous result about the prescription of the third fundamental form:\\
\\
\indent {\bf Corollary 3.4} {\it For every couple of metrics $g_{+}$ and $g_{-}$ on $S$ with curvature less than $-1$, there exists a globally hyperbolic convex compact $AdS_{3}$ manifold $K \cong S\times [0,1]$, such that the third fundamental forms on the boundary components are $g_{+}$ and $g_{-}$.}\\
\\
We outline here the main steps of the proof for the convenience of the reader. \\
\indent The first observation to be done is that Corollary \ref{main} is equivalent to proving that there exists a GHMC $AdS_{3}$ manifold $M$ containing a future-convex space-like surface isometric to $(S, g_{-})$ and a past-convex space-like surface isometric to $(S, g_{+})$. Adapting the work of Labourie (\cite{Lab}) to this Lorentzian setting, we prove that the space of isometric embeddings $I(S,g_{\pm})^{\pm}$ of $(S, g_{\pm})$ into a GHMC $AdS_{3}$ manifold as a future-convex (or past-convex) space-like surface is a manifold of dimension $6\tau -6$. On the other hand, by the work of Mess (\cite{Mess}), the space of GHMC $AdS_{3}$ structures is parameterised by two copies of Teichm\"uller space, hence a manifold of dimension $12\tau -12$. This allows us to translate our original question into a question about the existence of an intersection between subsets in $\dT$. More precisely, we will define in Section \ref{parametrization} two maps
\[
	\phi^{\pm}_{g_{\pm}}: I(S, g_{\pm})^{\pm} \rightarrow \dT
\]
sending an isometric embedding of $(S, g_{\pm})$ to the holonomy of the GHMC $AdS_{3}$ manifold containing it. Corollary \ref{main} is then equivalent to the following:\\
\\
\indent {\bf Theorem 3.2} {\it For every couple of metrics $g_{+}$ and $g_{-}$ on $S$ with curvature less than $-1$, we have
\[
	\phi^{+}_{g_{+}}(I(S,g_{+})^{+}) \cap \phi^{-}_{g_{-}}(I(S,g_{-})^{-}) \ne \emptyset \  . 
\]}

In order to prove this theorem we will use tools from topological intersection theory, which we recall in Section \ref{intersectiontheory}. For instance, Theorem \ref{mainteo} is already known to hold under particular hypothesis on the curvatures (\cite{landslide2}), hence we only need to check that the intersection persists when deforming one of the two metrics on the boundary, as the space of smooth metrics with curvature less than $-1$ is connected (see e.g. Lemma 2.3 in \cite{casofuchsiano}). More precisely, given any smooth paths of metrics $g^{\pm}_{t}$ with curvature less than $-1$, we will define the manifolds
\[
	 W^{\pm}=\bigcup_{t \in [0,1]} I(S, g^{\pm}_{t})^{\pm}
\]
and the maps
\[
	\Phi^{\pm}: W^{\pm} \rightarrow \dT
\]
with the property that the restrictions of $\Phi^{\pm}$ to the two boundary components coincide with $\phi_{g_{0}^{\pm}}^{\pm}$ and $\phi_{g_{1}^{\pm}}^{\pm}$. We will then prove the following:\\
\\
\indent {\bf Proposition 5.1} {\it The maps $\Phi^{\pm}$ are smooth.}\\
\\
\indent Hence, we will have the necessary regularity to apply tools from intersection theory. In particular, we can talk about transverse maps and under this condition we can define the intersection number (mod $2$) of the maps $\phi_{g_{+}}^{+}$ and $\phi_{g_{-}}^{-}$ as the cardinality (mod $2$), if finite, of $(\phi_{g_{+}}^{+} \times \phi_{g_{-}}^{-})^{-1}(\Delta)$, where
\[
	 \phi_{g_{+}}^{+} \times \phi_{g_{-}}^{-}: I(S, g_{+})^{+} \times I(S, g_{-})^{-} \rightarrow (\T(S))^{2} \times (\T(S))^{2}
\]
and $\Delta$ is the diagonal in $(\T(S))^{2} \times (\T(S))^{2}$. We will compute explicitly this intersection number (see Section \ref{proofmainthm}) under particular hypothesis on the curvatures of $g_{+}$ and $g_{-}$: the reason for this being that the transversality condition is in general difficult to check when the metrics do not have constant curvature. It turns out that in that case the intersection number is $1$.\\
\indent We then start to deform one of the two metrics and check that an intersection persists. Here, one has to be careful that, since the maps are defined on non-compact manifolds, the intersection does not escape to infinity. This is probably the main technical part of the paper and requires results about the convergence of isometric embeddings (Corollary \ref{convergenza}), estimates in Anti-de Sitter geometry (Lemma \ref{stimalaminazioni}) and results in Teichm\"uller theory (Lemma \ref{compactmetric}). In particular, applying these tools, we prove \\
\\
\indent {\bf Proposition 5.13} {\it For every metric $g^{-}$ and for every smooth path of metrics $\{g_{t}^{+}\}_{t \in [0,1]}$ on S with curvature less than $-1$, the set $(\Phi^{+}\times \phi^{-})^{-1}(\Delta)$ is compact}\\
\\
\indent This guarantees that when deforming one of the two metrics the variation of the intersection locus is always contained in a compact set.  The proof of Theorem \ref{mainteo} then follows applying standard argument of topological intersection theory. \\
\\
\indent In Section \ref{halfholonomy}, we study the map
\[
	p_{1} \circ \Phi^{+}: W^{+} \rightarrow \T(S) \ ,
\]
where $p_{1}: \dT \rightarrow \T(S)$ is the projection onto the left factor. The main result we obtain is the following: \\
\\
{\bf Proposition 6.1} {\it Let $g$ be a metric on $S$ with curvature less than $-1$ and let $h$ be a hyperbolic metric on $S$. Then there exists a GHMC $AdS_{3}$ manifold $M$ with left metric isotopic to $h$ containing a past-convex space-like surface isometric to $(S,g)$.}\\
\\
\indent This is proved by showing that $p_{1} \circ \phi^{+}_{g}$ is proper of degree $1$ (mod $2$). Again, we are able to compute explicitly the degree of the map when $g$ has constant curvature and the general statement then follows since for any couple of metrics $g$ and $g'$ with curvature less than $-1$, the maps $p_{1}\circ \phi_{g}$ and $p_{1} \circ \phi_{g'}$ are connected by a proper cobordism.  

\subsection*{Outline of the paper} In Section $\ref{Anti-de Sitter and GHMC Anti-de Sitter manifolds}$ we give a brief introduction to Anti-de Sitter space. We then describe a parameterisation of the space of GHMC $AdS_{3}$ structures in Section \ref{parametrization}. In Section \ref{embeddings} we study the space of isometric embeddings.
In Section \ref{intersectiontheory} we recall the main tools of topological intersection theory. Section \ref{studymaps} contains the most technical proofs: in particular we prove the smoothness and properness of the maps $\Phi^{\pm}$ and Proposition \ref{compact} . Section \ref{halfholonomy} deals with Proposition \ref{halfhol}. The main theorem (Theorem \ref{mainteo}) is proved in Section \ref{proofmainthm}. 

\section{Anti-de Sitter space}\label{Anti-de Sitter and GHMC Anti-de Sitter manifolds}
The $3$-dimensional Anti-de Sitter space $AdS_{3}$ is the Lorentzian analogue of hyperbolic space, i.e. it is the local model of Lorentzian $3$-manifolds with constant sectional curvature $-1$. In this section we describe a geometric model of $AdS_{3}$ as interior of a quadric in the real projective space and illustrate some of its features. We then introduce $AdS_{3}$ manifolds and the notion of globally hyperbolicity. The main references for this material are \cite{Mess} and \cite{folKsurfaces}.\\
\\
\indent Consider in $\R^{4}$ the bilinear form of signature $(2,2)$
\[
	\langle x , y \rangle_{2,2}= x_{1}y_{1}+x_{2}y_{2}-x_{3}y_{3}-x_{4}y_{4} \ \ \ \  \ \ \ x,y \in \R^{4} \ .
\]
We denote with $Q$ the hyperboloid
\[
	Q=\{ x \in \R^4 \ | \ \langle x, x \rangle_{2,2}=-1 \ \} \ .
\]
The restriction of the bilinear form $\langle \cdot, \cdot \rangle_{2,2}$ to the tangent spaces of $Q$ induces a Lorentzian metric on $Q$ with constant sectional curvature $-1$. Geodesics and totally geodesic planes are obtained intersecting $Q$ with planes and hyperplanes of $\R^{4}$ through the origin. Moreover, the group of orientation and time-orientation isometry of $Q$ is the connected component of $SO(2,2)$ containing the identity.\\

We define the Anti-de Sitter space $AdS_{3}$ as the image of the projection of $Q$ into $\R\Pp^{3}$. More precisely, if we denote with $\pi: \R^{4}\setminus \{0\} \rightarrow \R\Pp^{3}$ the canonical projection, we define
\[
	AdS_{3}=\pi(\{ x \in \R^{4} \ | \ \langle x, x \rangle_{2,2}<0 \ \}) \ .
\]
It can be easily verified that $\pi: Q \rightarrow AdS_{3}$ is a double cover, hence we can endow $AdS_{3}$ with the unique Lorentzian structure that makes $\pi$ a local isometry. It then follows by the definition that geodesic and totally geodesic planes of $AdS_{3}$ are obtained intersecting $AdS_{3}$ with projective lines and planes. \\

It is then natural to define the boundary at infinity of $AdS_{3}$ as 
\[
	\partial_{\infty}AdS_{3}=\pi(\{ x \in R^{4} \ | \ \langle x,x \rangle_{2,2}=0 \ \} \ .
\]
It can be easily verified that $\partial_{\infty}AdS_{3}$ coincides with the image of the Segre embedding
\[
	s: \R\Pp^{1} \times \R\Pp^{1} \rightarrow \R\Pp^{3} \ ,
\]
hence the boundary at infinity of Anti-de Sitter space is a double-ruled quadric homeomorphic to $S^{1} \times S^{1}$. We will talk about left and right ruling in order to distinguish the two rulings. This homeomorphism can be also described geometrically in the following way. Fix a totally geodesic space-like plane $P_{0}$ in $AdS_{3}$. The boundary at infinity of $P_{0}$ is a circle. Let $\xi \in \partial_{\infty}AdS_{3}$. There exists a unique line of the left ruling $l_{\xi}$ and a unique line of the right ruling $r_{\xi}$ passing through $\xi$. The identification between $\partial_{\infty}AdS_{3}$ and $S^{1} \times S^{1}$ induced by $P_{0}$ associates to $\xi$ the intersection points $\pi_{l}(\xi)$ and $\pi_{r}(\xi)$ between $l_{\xi}$ and $r_{\xi}$ with the boundary at infinity of $P_{0}$. These two maps
\[
	\pi_{l}: \partial_{\infty}AdS_{3} \rightarrow S^{1} \ \ \ \ \ \ \pi_{r}: \partial_{\infty}AdS_{3} \rightarrow S^{1}
\]
are called left and right projections. \\
\indent The action of an orientation and time-orientation preserving isometry of $AdS_{3}$ extends continuously to the boundary at infinity and it is projective on the two rulings, thus giving an identification between $SO_{0}(2,2)$ and $\dPSL$.\\

The projective duality between points and planes of $\R\Pp^{3}$ induces a duality in $AdS_{3}$ between points and totally geodesic space-like planes. This duality induces then a duality between smooth, space-like convex surfaces in $AdS_{3}$. Namely, let $\tilde{S}\subset AdS_{3}$ be a convex space-like surface. Denote by $\tilde{S}^{*}$ the set of points which are duals to the tangent planes of $\tilde{S}$. The relation between $\tilde{S}$ and $\tilde{S}^{*}$ is summarised in the following lemma.

\begin{lemma}[see e.g. Section 11 of \cite{folKsurfaces}] \label{duality} Let $\tilde{S}\subset AdS_{3}$ be a smooth, space-like surface with curvature $\K <-1$. Then
\begin{enumerate}[(i)]
	\item the dual surface $\tilde{S}^{*}$ is smooth and locally strictly convex;
	\item the pull-back of the induced metric on $\tilde{S}^{*}$ through the duality map is the third fundamental form\footnotemark of $\tilde{S}$;
	\item if $\K$ is constant, the dual surface $\tilde{S}^{*}$ has curvature $\K^{*}=-\frac{\K}{\K+1}$.
\end{enumerate}
\end{lemma}
\footnotetext{We recall that the third fundamental form of an embedded surface $S$ with induced metric $I$ and shape operator $B$ is the bilinear form on $TS$ defined by $III(X,Y)=I(BX,BY)$ for every vector field $X$ and $Y$ on $S$.}  

An $AdS_{3}$ manifold $M$ is a $3$-manifold endowed with a Lorentzian metric locally isometric to $AdS_{3}$. We say that $M$ is globally hyperbolic (GH) if it contains an embedded space-like surface (called Cauchy surface) which intersects every time-like line exactly once. $M$ is said to be spacially compact (C) if it admits a compact Cauchy surface.  This leads to the following definition:
\begin{defi} A globally hyperbolic maximal compact (GHMC) $AdS_{3}$ manifold is a globally hyperbolic $AdS_{3}$ manifold $M$ containing a compact Cauchy surface with the following property: if $i: M \rightarrow M'$ is an isometric embedding of $M$ into another $GHC$ $AdS_{3}$ manifold $M'$ sending a Cauchy surface into a Cauchy surface, then $i$ is an isometry.
\end{defi}

If $M$ is a GHMC $AdS_{3}$-manifold, then its universal cover can be identified with a subset $D$ of $AdS_{3}$ which can be roughly described as follows: the closure of $D$ intersects the boundary at infinity of $AdS_{3}$ along a curve $\rho$ and the interior of $D$ is the set of points such that the boundary at infinity of the dual planes are disjoint from $\rho$. This subset $D=D(\rho)$ is called the domain of dependence of the curve $\rho$. 
\begin{oss} The above description can be made more precise, taking into account the causal property of the curve $\rho$ and its regularity (see e.g. \cite{Mess})  but we will not use these notions in the rest of the paper.
\end{oss}

The duality between smooth space-like surfaces in $AdS_{3}$ induces a similar duality between smooth space-like surfaces in a GHMC $AdS_{3}$ manifold. More precisely, let $S\subset M$ be a smooth, space-like, strictly-convex surface in a GHMC $AdS_{3}$ manifold. The lift of $S$ to the universal cover of $M$ can be identified with a surface $\tilde{S}$ in $AdS_{3}$, invariant under the action of the fundamental group of $S$. The dual surface $\tilde{S}^{*}$ is also invariant, so it corresponds to a surface $S^{*}$ in $M$. Clearly, since the fundamental group of $S$ acts by isometries, an analogue of Lemma \ref{duality} holds.

\section{Equivariant isometric embeddings}\label{embeddings}
 Let $S$ be a connected, compact, oriented surface of genus $\tau \geq 2$ and let $g$ be a Riemannian metric on $S$ with curvature less than $-1$. An isometric equivariant embedding of $S$ into $AdS_{3}$ is given by a couple $(f,\rho)$, where $f:\tilde{S}\rightarrow AdS_{3}$ is an isometric embedding of the universal Riemannian cover of $S$ into $AdS_{3}$ and $\rho$ is a representation of the fundamental group of $S$ into $\dPSL$ such that
\[
	f(\gamma x)=\rho(\gamma)f(x) \ \ \  \forall \ \gamma\in \pi_{1}(S) \ \ \ \forall \ x\in \tilde{S} \ .
\]
 The group $\dPSL$ acts on a couple $(f,\rho)$ by post-composition on the embedding and by conjugation on the representation. We denote by $I(S,g)$ the set of isometric embeddings of $S$ into $AdS_{3}$ modulo the action of $\dPSL$.
\\ 
Also in an Anti-de Sitter setting, an analogue of the Fundamental Theorem for surfaces in the Euclidean space holds:

\begin{teo}There exists an isometric embedding of $(S,g)$ into an $AdS_{3}$ manifold if and only if it is possible to define a $g$-self-adjoint operator $b:TS\rightarrow TS$ satisfying 
	\begin{align*}
	&\det(b)=-\kappa-1 \ \ \  \ \ \ &\text{Gauss equation} \\
	&d^{\nabla}b=0 \ \ \ \ \ \ &\text{Codazzi equation}
	\end{align*}
Moreover, the operator $b$ determines the isometric embedding uniquely, up to global isometries.
\end{teo}

 This theorem enables us to identify $I(S,g)$ with the space of solutions of the Gauss-Codazzi equations, which can be studied using the classical techniques of elliptic operators.
\begin{lemma}\label{manifold} The space $I(S,g)$ is a manifold of dimension $6\tau -6$.  
\end{lemma}
\begin{proof} We can mimic the proof of Lemma $3.1$ in \cite{Lab}. Consider the sub-bundle $\mathpzc{F}^{g}\subset \Sym(TS)$ over $S$ of symmetric operators $b:TS\rightarrow TS$ satisfying the Gauss equation. We prove that the operator
\[
	d^{\nabla}:\Gamma^{\infty}(\mathpzc{F}^{g}) \rightarrow \Gamma^{\infty}(\Lambda^{2}TS\otimes TS)
\]
is elliptic of index $6\tau-6$, equal to the dimension of the kernel of its linearization.\\
Let $J_{0}$ be the complex structure induced by $g$. For every $b\in \Gamma^{\infty}({\mathpzc{F}^{g}})$, the operator
\[
	J=\frac{J_{0}b}{\sqrt{\det(b)}} 
\]
defines a complex structure on $S$. In particular we have an isomorphism 
	\begin{align*}
	F:\Gamma^{\infty}(\mathpzc{F}^{g}) &\rightarrow \mathpzc{A} \\
		b &\mapsto \frac{J_{0}b}{\sqrt{\det(b)}}
	\end{align*}
between smooth sections of the sub-bundle $\Gamma^{\infty}(\mathpzc{F}^{g})$ and the space $\mathpzc{A}$ of complex structures on $S$, with inverse
	\begin{align*}
	F^{-1}:\mathpzc{A} &\rightarrow \Gamma^{\infty}(\mathpzc{F}^{g})  \\
		J &\mapsto -\sqrt{-\K-1}J_{0}J \ .
	\end{align*}
This allows us to identify the tangent space of $\Gamma(\mathpzc{F}^{g})$ at $b$ with the tangent space of $\mathpzc{A}$ at $J$, which is the vector space of operators $\dot{J}: TS \rightarrow TS$ such that $\dot{J}J+J\dot{J}=0$. Under this identification the linearization of $d^{\nabla}$ is given by 
\[
	L(\dot{J})=-J_{0}(d^{\nabla}\dot{J}) \ .
\]
We deduce that $L$ has the same symbol and the same index of the operator $\overline{\partial}$, sending quadratic differentials to vector fields. Thus $L$ is elliptic with index $6\tau-6$. \\
To conclude we need to show that its cokernel is empty, or, equivalently, that its adjoint $L^{*}$ is injective. If we identify $\Lambda^{2}TS\otimes TS$ with $TS$ using the metric $g$, the adjoint operator $L^{*}$ is given by (see Lemma 3.1 in \cite{Lab} for the computation)
\[
	(L^{*}\psi)(u)=-\frac{1}{2}(\nabla_{J_{0}u}\psi+J\nabla_{J_{0}Ju}\psi) \ .
\]
The kernel of $L^{*}$ consists of all the vector fields $\psi$ on $S$ such that for every vector field $u$ 
\[
	J\nabla_{u}\psi=-\nabla_{J_{0}JJ_{0}u}\psi \ .
\]
We can interpret this equation in terms of intersection of pseudo-holomorphic curves: the Levi-Civita connection $\nabla$ induces a decomposition of $T(TS)$ into a vertical $V$ and a horizontal $H$ sub-bundle. We endow $V$ with the complex structure $J$, and $H$ with the complex structure $-J_{0}JJ_{0}$. In this way, the manifold $TS$ is endowed with an almost-complex structure and the graph of $\psi$ is a pseudo-holomorphic curve. Since pseudo-holomorphic curves have positive intersections, if the graph of $\psi$ did not coincide with the graph of the null section, their intersection would be positive. On the other hand, it is well-known that this intersection coincides with the Euler characteristic of $S$, which is negative. Hence, we conclude that $\psi$ is identically zero and that $L^{*}$ is injective. 
\end{proof}

 Similarly, we obtain the following result:
\begin{lemma}\label{manifold2}Let $\{g_{t}\}_{t\in [0,1]}$ be a differentiable curve of metrics with curvature less than $-1$. The set
\[
	W=\bigcup_{t\in [0,1]}I(S,g_{t}) 
\]
is a manifold with boundary of dimension $6\tau-5$.
\end{lemma} 
\begin{proof} Again we can mimic the proof of Lemma $3.2$ in \cite{Lab}. Consider the sub-bundle $\mathpzc{F}\subset \Sym(TS)$ over $S\times[0,1]$ of symmetric operators, whose fiber over a point $(x,t)$ consists of the operators $b:TS \rightarrow TS$, satisfying the Gauss equation with respect to the metric $g_{t}$. The same reasoning as for the previous lemma shows that
\[
	d^{\nabla}:\Gamma^{\infty}(\mathpzc{F}) \rightarrow \Gamma^{\infty}(\Lambda^{2}TS\otimes TS)
\]
is Fredholm of index $6\tau-5$. Since $W=(d^{\nabla})^{-1}(0)$, the result follows from the implicit function theorem for Fredholm operators.
\end{proof}

 Let $N$ be a GHMC $AdS_{3}$ manifold endowed with a time orientation, i.e. a nowhere vanishing time-like vector field. Let $S$ be a convex embedded surface in $N$. We say that $S$ is past-convex (resp. future-convex), if its past (resp. future) is geodesically convex. We will use the convention to compute the shape operator of $S$ using the future-directed normal. With this choice if $S$ is past-convex (resp. future-convex) then it has strictly positive (resp. strictly negative) principal curvatures. 

\begin{defi} We will denote with $I(S,g)^{+}$ and $I(S,g)^{-}$ the spaces of equivariant isometric embeddings of $S$ as a past-convex and future-convex surface, respectively.
\end{defi}

\section{Parameterisation of GHMC $AdS_{3}$ manifolds}\label{parametrization}

 In his pioneering work (\cite{Mess}), Mess studied the geometry of GHMC $AdS_{3}$ manifolds, discovering many connections with Teichm\"uller theory. We recall here some of his results, which we are going to use further. \\

 Let $N$ be a GHMC $AdS_{3}$ spacetime, i.e. $N$ is endowed with an orientation and a time-orientation. It contains a space-like Cauchy surface, which is a closed surface $S$ of genus $\tau \geq 2$. It follows that $N$ is diffeomorphic to the product $S\times \R$. It contains a convex core, i.e. a minimal convex subset homotopy equivalent to $N$, which can be either a totally geodesic surface or a topological submanifold homeomorphic to $S\times [0,1]$, whose boundary components are space-like surfaces, naturally endowed with a hyperbolic metric, pleated along measured geodesic laminations.\\
\indent The holonomy representation $\rho$ of the fundamental group $\pi_{1}(N)\cong \pi_{1}(S)$ into the group $\dPSL$ of orientation and time-orientation preserving isometries of $AdS_{3}$ induced by the $AdS_{3}$-structure can be split, by projecting to each factor, into two representations 
\[
	\rho_{l}=p_{1}\circ \rho: \pi_{1}(S) \rightarrow \PSL(2,\R) \ \ \ \rho_{r}=p_{2}\circ \rho: \pi_{1}(S) \rightarrow \PSL(2,\R)
\]
called left and right representations. Mess proved that these have Euler class $|e(\rho_{l})|=|e(\rho_{r})|=2g-2$ and, using Goldman's criterion (\cite{Goldman}), he concluded that they are discrete and faithful representations, and thus their classes represent elements of Teichm\"uller space. Moreover, every couple of points in Teichm\"uller space can be realised uniquely in this way, thus giving a parameterisation of the set of GHMC $AdS_{3}$-structures on $S\times \R$ up to isotopy by the product of two copies of the Teichm\"uller space of $S$ (Prop. 19 and Prop. 20 in \cite{Mess}). We say that a GHMC $AdS_{3}$ manifold $N$ is Fuchsian if its left and right representations represent the same class in Teichm\"uller space. This happens if and only if the convex core of $N$ is a totally geodesic surface.\\
\indent Moreover, the left and right hyperbolic metrics corresponding to the left and right representations can be constructed explicitly starting from space-like surfaces embedded in $N$. Mess gave a description in a non-smooth setting using the upper and lower boundary of the convex core of $N$ as space-like surfaces. More precisely, if $m^{\pm}$ are the hyperbolic metrics on the upper and lower boundary of the convex core and $\lambda^{\pm}$ are the measured geodesic lamination along which they are pleated, the left and right metrics $h_{l}$ and $h_{r}$ are related to $m^{\pm}$ by an earthquake along $\lambda^{\pm}$:
\[
	h_{l}=E_{l}^{\lambda^{+}}(m_{+})=E_{r}^{\lambda^{-}}(m_{-}) \ \ \ \ \ \ 
	h_{r}=E_{r}^{\lambda^{+}}(m_{+})=E_{l}^{\lambda^{-}}(m_{-}) \ .
\]

Later, this description was extended (\cite{SchKra}), thus obtaining explicit formulas for the left and right metric, in terms of the induced metric $I$, the complex structure $J$ and the shape operator $B$ of any strictly negatively curved smooth space-like surface $S$ embedded in $N$. The construction goes as follows. We fix a totally geodesic space-like plane $P_{0}$. Let $\tilde{S}\subset AdS_{3}$ be the universal cover of $S$. Let $\tilde{S}'\subset U^{1}AdS_{3}$ be its lift into the unit tangent bundle of $AdS_{3}$ and let $p:\tilde{S}' \rightarrow \tilde{S}$ be the canonical projection. For any point $(x,v)\in \tilde{S}'$, there exists a unique space-like plane $P$ in $AdS_{3}$ orthogonal to $v$ and containing $x$. We define two natural maps $\Pi_{\infty, l}$ and $\Pi_{\infty,r}$ from $\partial_{\infty}P$ to $\partial_{\infty}P_{0}$, sending a point $x\in \partial_{\infty}P$ to the intersection between $\partial_{\infty}P_{0}$ and the unique line of the left or right foliation of $\partial_{\infty}AdS_{3}$ containing $x$. Since these maps are projective, they extend to hyperbolic isometries $\Pi_{l}, \Pi_{r}:P \rightarrow P_{0}$. Identifying $P$ with the tangent space of $\tilde{S}$ at the point $x$, the pull-backs of the hyperbolic metric on $P_{0}$ by $\Pi_{l}$ and by $\Pi_{r}$ define two hyperbolic metrics on $\tilde{S}$ 
\[
	h_{l}=I((E+JB)\cdot, (E+JB)\cdot)  \ \ \ \ \text{and} \ \ \ \ h_{r}=I((E-JB)\cdot, (E-JB)\cdot)\ .
\]
The isotopy classes of the corresponding metrics on $S$ do not depend on the choice of the space-like surface $S$ and their holonomies are precisely $\rho_{l}$ and $\rho_{r}$, respectively (Lemma 3.16 in \cite{SchKra}). 
\\

 This parameterisation enables us to formulate our original question about the prescription of the metrics on the boundary of a compact $AdS_{3}$ manifold in terms of existence of an intersection of particular subsets of $\dT$.

\begin{defi}\label{defimap1met} Let $g$ be a metric on $S$ with curvature $\K<-1$. We define the maps
	\begin{align*}
	\phi^{\pm}_{g}: I(&S,g)^{\pm} \rightarrow \dT \\
			&b \mapsto (h_{l}(g,b), h_{r}(g,b)):=(g((E+Jb)\cdot, (E+Jb)\cdot), g((E-Jb)\cdot, (E-Jb)\cdot))
	\end{align*}
associating to every isometric embedding of $(S,g)$ the left and right metric of the GHMC $AdS_{3}$ manifold containing it. 
\end{defi}

We recall that we use the convention to compute the shape operator using always the future-oriented normal. In this way, the above formulas hold for both future-convex and past-convex surfaces, without changing the orientation of the surface $S$.\\

We will prove (in Section \ref{proofmainthm}) the following fact, which is the main theorem of the paper:

\begin{teo}\label{mainteo} For every couple of metrics $g_{+}$ and $g_{-}$ on $S$ with curvature less than $-1$, we have
\[
	\phi^{+}_{g_{+}}(I(S,g_{+})^{+}) \cap \phi^{-}_{g_{-}}(I(S,g_{-})^{-}) \ne \emptyset \  . 
\]

 Therefore, there exists a GHMC $AdS_{3}$ manifold containing a past-convex space-like surface isometric to $(S,g_{+})$ and a future-convex space-like surface isometric to $(S, g_{-})$.
\end{teo}
\begin{cor}\label{main} For every couple of metrics $g_{+}$ and $g_{-}$ on $S$ with curvature less than $-1$, there exists a globally hyperbolic convex compact $AdS_{3}$ manifold $K\cong S\times [0,1]$, whose induced metrics on the boundary are exactly $g_{\pm}$.
\end{cor}

If we apply the previous corollary to the dual surfaces (introduced in Section \ref{Anti-de Sitter and GHMC Anti-de Sitter manifolds}), we obtain an analogue result about the prescription of the third fundamental form:
\begin{cor}\label{III} For every couple of metrics $g_{+}$ and $g_{-}$ on $S$ with curvature less than $-1$, there exists a compact $AdS_{3}$ manifold $K\cong S\times [0,1]$, whose induced third fundamental forms on the boundary are exactly $g_{\pm}$. 
\end{cor}

\section{Topological intersection theory}\label{intersectiontheory}
\indent As outlined in the Introduction, the main tool used in the proof of the main theorem is the intersection theory of smooth maps between manifolds, which is developed for example in \cite{trasversalità}. We recall here the basic constructions and the fundamental results.\\
\\
\indent If not otherwise stated, all manifolds considered in this section are non-compact without boundary.\\
\indent  Let $X$ and $Z$ be manifolds of dimension $m$ and $n$, respectively and let $A$ be a closed submanifold of $Z$ of codimension $k$. Suppose that $m-k \geq 0$. We say that a smooth map $f:X \rightarrow Z$ is transverse to $A$ if for every $z\in \Ima(f)\cap A$ and for every $x\in f^{-1}(z)$ we have
\[
	df(T_{x}X) + T_{z}A=T_{z}Z \ .
\]
Under this hypothesis, $f^{-1}(A)$ is a submanifold of $X$ of codimension $k$. \\
When $k=m$ and $f^{-1}(A)$ consists of a finite number of points we define the intersection number between $f$ and $A$ as
\[
	\Im(f,A):= |f^{-1}(A)| \ \ \ \ \ \ \ (\text{mod}\ 2) \ .
\]

\begin{oss}\label{defdegree} When $A$ is a point $p\in Z$, $f$ is transverse to $p$ if and only if $p$ is a regular value for $f$. Moreover, if $f$ is proper, $f^{-1}(p)$ consists of a finite number of points and the above definition coincides with the classical definition of degree (mod $2$) of a smooth and proper map.
\end{oss}

 We say that two smooth maps $f: X \rightarrow Z$ and $g: Y \rightarrow Z$ are transverse if the map
\[
	f\times g : X \times Y \rightarrow Z \times Z
\]
is transverse to the diagonal $\Delta \subset Z\times Z$. Notice that if $\Ima(f) \cap \Ima(g)=\emptyset$, then $f$ and $g$ are transverse by definition.\\
Suppose now that $2\dim X=2\dim Y=\dim Z$. Moreover, suppose that the maps $f:X \rightarrow Z$ and $g: Y \rightarrow Z$ are transverse and the preimage $(f\times g)^{-1}(\Delta)$ consists of a finite number of points. We define the intersection number between $f$ and $g$ as
\[
	\Im(f,g):=\Im(f\times g, \Delta)=|(f\times g)^{-1}(\Delta)| \ \ \ \ \ \ \ (\text{mod}\  2) \ .
\]
It follows by the definition that if $\Im(f,g)\ne 0$ then $\Ima(f)\cap \Ima(g) \neq \emptyset$.\\

 One important feature of the intersection number that we will use further is the invariance under cobordism. We say that two maps $f_{0}:X_{0} \rightarrow Z$ and $f_{1}: X_{1} \rightarrow Z$ are cobordant if there exists a manifold $W$ and a smooth function $F:W \rightarrow Z$ such that $\partial W=X_{0}\cup X_{1}$ and $F_{|_{X_{i}}}=f_{i}$. 

\begin{prop}\label{invcobgen} Let $W$ be a non-compact manifold with boundary $\partial W=X_{0}\cup X_{1}$. Let $H: W \rightarrow Z$ be a smooth map and denote by $h_{i}$ the restriction of $H$ to the boundary component $X_{i}$ for $i=0,1$. Let $A \subset Z$ be a closed submanifold. 
Suppose that
\begin{enumerate}[(i)]
\begin{item} $\codim A=\dim X_{i}$;
\end{item}
\begin{item} $H$ is transverse to $A$;
\end{item}
\begin{item} $H^{-1}(A)$ is compact.
\end{item}
\end{enumerate}
Then $\Im(h_{0},A)=\Im(h_{1},A)$.
\end{prop}
\begin{proof}By hypothesis the pre-image $H^{-1}(A)$ is a compact, properly embedded $1$-manifold, i.e. it is a finite disjoint union of circles and arcs with ending points on a boundary component of $W$. This implies that $h_{0}^{-1}(A)$ and $h_{1}^{-1}(A)$ have the same parity. 
\end{proof}

 In particular, we deduce the following result about the intersection number of two maps:
\begin{cor} \label{invcob} Let $W$ be a non-compact manifold with boundary $\partial W=X_{0}\cup X_{1}$. Let $F: W \rightarrow Z$ be a smooth map and denote by $f_{i}$ the restriction of $F$ to the boundary component $X_{i}$ for $i=0,1$. Let $g:Y \rightarrow Z$ be a smooth map. 
Suppose that
\begin{enumerate}[(i)]
\begin{item} $2\dim X_{i}=2\dim Y=\dim Z$;
\end{item}
\begin{item} $F$ and $g$ are transverse;
\end{item}
\begin{item} $(F\times g)^{-1}(\Delta)$ is compact.
\end{item}
\end{enumerate}
Then $\Im(f_{0},g)=\Im(f_{1},g)$.
\end{cor}
\begin{proof} Apply the previous proposition to the map $H=F\times g: W\times Y \rightarrow Z\times Z$ and to the submanifold $A=\Delta$, the diagonal of $Z\times Z$.
\end{proof}

 The hypothesis of transversality in the previous propositions is not restrictive, as it is always possible to perturb the maps involved on a neighbourhood of the set on which transversality fails:
\begin{teo}[Theorem p.72 in \cite{trasversalità}] \label{perturb} Let $h: W \rightarrow Z$ be a smooth map between manifolds, where only $W$ has boundary. Let $A$ be a closed submanifold of $Z$. Suppose that $h$ is transverse to $A$ on a closed set $C\subset W$. Then there exists a smooth map $\tilde{h}:W \rightarrow Z$ homotopic to $h$ such that $\tilde{h}$ is transverse to $A$ and $\tilde{h}$ agrees with $h$ on a neighbourhood of $C$.
\end{teo} 

 Now the question arises whether the intersection number depends on the particular perturbation of the map that we obtain when applying Theorem \ref{perturb}.
\begin{prop}\label{invpert}Let $h:X\rightarrow Z$ be a smooth map between manifolds. Let $A$ be a submanifold of $Z$, whose codimension equals the dimension of $X$. Suppose that $h^{-1}(A)$ is compact. Let $\tilde{h}$ and $\tilde{h}'$ be perturbations of $h$, which are transverse to $A$ and coincide with $h$ outside the interior part of a compact set $B$ containing $h^{-1}(A)$. Then 
\[
 \Im(\tilde{h},A)=\Im(\tilde{h}',A) \ .
\]
\end{prop}
\begin{proof}Let $\tilde{H}:W=X\times [0,1]\rightarrow Y$ be an homotopy between $\tilde{h}$ and $\tilde{h}'$ such that for every $x\in (X\setminus B)\times [0,1]$ we have $\tilde{H}(x,t)=h(x)$ . Notice that $\tilde{H}^{-1}(A)$ is compact. Up to applying Theorem \ref{perturb} to the closed set $C=(X\setminus B)\times [0,1]\cup \partial W$, we can suppose that $\tilde{H}$ is transverse to $A$. By Proposition \ref{invcobgen}, we have that $\Im(\tilde{h}_{0},A)=\Im(\tilde{h}_{1},A)$ as claimed.  
\end{proof}  

 Moreover, in particular circumstances, we can actually obtain a $1-1$ correspondence between the points of $h_{0}^{-1}(A)$ and $h_{1}^{-1}(A)$. The following proposition will not be used for the proof of the main result of the paper, but it might be a useful tool to prove the uniqueness part of the question addressed in this paper, as explained in Remark \ref{unique?}.
\begin{prop}\label{unico} Under the same hypothesis as Proposition \ref{invcobgen}, suppose that the cobordism $(W, H)$ between $h_{0}$ and $h_{1}$ satisfies the following additional properties:
\begin{enumerate}[(i)] 
\begin{item} $W$ fibers over the interval $[0,1]$ with fiber $X_{t}$;
\end{item}
\begin{item} the restriction $h_{t}$ of $H$ at each fiber is tranverse to $A$ .
\end{item}
\end{enumerate}
Then $|h_{0}^{-1}(A)|=|h_{1}^{-1}(A)|$.
\end{prop}
\begin{proof}It is sufficient to show that in $H^{-1}(A)$ there are no arcs with ending points in the same boundary component. By contradiction, let $\gamma$ be an arc with ending point in $X_{0}$. Define
\[
	t_{0}=\sup\{t \in [0,1] \ | \ \gamma \cap X_{t} \ne \emptyset \} \ .
\]
A tangent vector $\dot{\gamma}$ at a point $p \in X_{t_{0}}\cap \gamma$ is in the kernel of the map 
\[
	d_{p}H: T_{(p,t_{0})}W  \rightarrow T_{q}(Z\times Z)/T_{q}(A) \ ,
\]
where $q=H(p)$. The contradiction follows by noticing that on the one hand $\dot{\gamma}$ is contained in the tangent space $T_{p}X_{t_{0}}$ by construction but on the other hand $d_{p}h_{t_{0}}: T_{p}X_{t_{0}} \rightarrow T_{q}(Z\times Z)/T_{q}(A)$ is an isomorphism by transversality. \\
A similar reasoning works when $\gamma$ has ending points in $X_{1}$.
\end{proof}

\section{Some properties of the maps $\phi^{\pm}$}\label{studymaps}
This section contains the most technical part of the paper. We summarise here briefly, for the convenience of the reader, what the main results of this section are.
\indent For every metric $g$ on $S$ with curvature less than $-1$ we have defined in Section \ref{parametrization} the maps $\phi_{g}^{\pm}$ which associate to every isometric embedding of $(S,g)$ into a GHMC $AdS_{3}$ manifold $M$ the class in Teichm\"uller space of the left and right metrics of $M$. It follows easily from Lemma \ref{manifold2} that for any couple of metrics $g$ and $g'$ with curvature less than $-1$ the maps $\phi^{\pm}_{g}$ and $\phi^{\pm}_{g'}$ are cobordant through a map $\Phi^{\pm}$. In this section we will define the maps $\Phi^{\pm}$ and will study some of its properties, which will enable us to apply the topological intersection theory described in the previous section. More precisely, the first step will consist of proving that the all the maps involved are smooth. This is the content of Proposition \ref{smoothness} and the proof will rely on the fact that the holonomy representation of a hyperbolic metric depends smoothly on the metric. Then we will deal with the properness of the maps $\Phi^{\pm}$ (Corollary \ref{properdef}) that will follow from a compacteness result of isometric embeddings (Corollary \ref{convergenza}). This will allow us also to have a control on the space where two maps $\phi_{g}$ and $\phi_{g'}$ intersect: when we deform one of the two metrics the intersection remains contained in a compact set (Proposition \ref{compact}). \\

 Recall that given a smooth path of metrics $\{g_{t}\}_{t \in [0,1]}$ on $S$ with curvature less than $-1$, the set
\[
	W^{\pm}=\bigcup_{t \in [0,1]}I^{\pm}(S,g_{t}) 
\]
is a manifold with boundary $\partial W^{\pm}=I(S, g_{0})^{\pm} \cup I(S, g_{1})^{\pm}$ of dimension $6\tau-5$ (Lemma \ref{manifold2}). We define the maps
	\begin{align*}
	\Phi^{\pm}: W^{\pm} &\rightarrow \dT\\
		b_{t} &\mapsto (h_{l}(g_{t}, b_{t}), h_{r}(g_{t}, b_{t})):=(g_{t}((E+Jb_{t})\cdot, (E+Jb_{t})\cdot), g_{t}((E-Jb_{t})\cdot, (E-Jb_{t})\cdot))
	\end{align*}
associating to an equivariant isometric embedding (identified with its Codazzi operator $b_{t}$) of $(S, g_{t})$ into $AdS_{3}$ the class in Teichm\"uller space of the left and right metrics of the GHMC $AdS_{3}$ manifold containing it. We remark that the restrictions of $\Phi^{\pm}$ to the boundary coincide with the maps $\phi^{\pm}_{g_{0}}$ and $\phi^{\pm}_{g_{1}}$ defined in Section \ref{parametrization}.\\
\\
\indent We deal first with the regularity of the maps.
\begin{prop}\label{smoothness}The functions $\Phi^{\pm}:W^{\pm} \rightarrow \dT$ are smooth.
\end{prop}
\begin{proof} Let $\mathpzc{M}_{S}$ be the set of hyperbolic metrics on $S$. We can factorise the map $\Phi^{\pm}$ as follows:
\[
	W^{\pm} \xrightarrow{\Phi'^{\pm}} \mathpzc{M}_{S}\times \mathpzc{M}_{S} \xrightarrow{\pi} \dT\\
\]
where $\Phi'^{\pm}$ associates to an isometric embedding of $(S, g_{t})$ (determined by an operator $b_{t}$ satisfying the Gauss-Codazzi equation) the couple of hyperbolic metrics $(g_{t}((E+J_{t}b_{t})\cdot, (E+J_{t}b_{t})\cdot), g_{t}((E-J_{t}b_{t})\cdot, (E-J_{t}b_{t})\cdot))$, and $\pi$ is the projection to the corresponding isotopy class, or, equivalently, the map which associates to a hyperbolic metric its holonomy representation. Since the maps $\Phi'^{\pm}$ are clearly smooth by definition, we just need to prove that the holonomy representation depends smoothly on the metric. Let $h$ be a hyperbolic metric on $S$. Fix a point $p \in S$ and a unitary frame $\{v_{1}, v_{2}\}$ of the tangent space $T_{p}S$. We consider the ball model for the hyperbolic plane and we fix a unitary frame $\{w_{1}, w_{2}\}$ of $T_{0}\mathbb{H}^{2}$. We can realise every element of the fundamental group of $S$ as a closed path passing through $p$. Let $\gamma$ be a path passing through $p$ and let $\{U_{i}\}_{i=0, \dots n}$ be a finite covering of $\gamma$ such that every $U_{i}$ is homeomorphic to a ball. We know that there exists a unique map $f_{0}: U_{0} \rightarrow B_{0}\subset \mathbb{H}^{2}$ such that
\[
	\begin{cases}
	f_{0}(p)=0 \\ 
	d_{p}f_{0}(v_{i})=w_{i}\\
	f_{0}^{*}g_{\mathbb{H}^{2}}=h \ .
	\end{cases}
\]
Then, for every $i\geq 1$ there exists a unique isometry $f_{i}: U_{i} \rightarrow B_{i}\subset \mathbb{H}^{2}$ which coincides with $f_{i-1}$ on the intersection $U_{i} \cap U_{i-1}$. Let $q=f_{n}(p) \in \mathbb{H}^{2}$. The holonomy representation sends the homotopy class of the path $\gamma$ to the isometry $I_{q}: \mathbb{H}^{2}\rightarrow \mathbb{H}^{2}$ such that $I_{q}(q)=0$. Moreover, its differential maps the frame $\{u_{i}=df_{n}(v_{i})\}$ to the frame $w_{i}$. The isometry $I_{q}$ depends smoothly on $q$ and on the frame $u_{i}$, which depend smoothly on the metric because each $f_{i}$ does.
\end{proof}

The next step is about the properness of the maps $\Phi^{\pm}$.  This will involve the study of sequences of isometric embeddings of a disc into a simply-connected spacetime, which have been extensively and profitably analysed in \cite{isomsch}. In particular, the author proved that, under reasonable hypothesis, a sequence of isometric embeddings of a disc into a simply-connected spacetime has only two possible behaviours: it converges $C^{\infty}$, up to subsequences, to an isometric embedding, or it is degenerate in a precise sense:

\begin{teo}[Theorem 5.6 in \cite{isomsch}]\label{degenere} Let $\tilde{f}_{n}: D\rightarrow X$  be a sequence of uniformly elliptic \footnotemark[1] immersions of a disc $D$ in a simply connected Lorentzian spacetime $(X, \tilde{g})$. Assume that the metrics $\tilde{f}_{n}^{*}\tilde{g}$ converge $C^{\infty}$ towards a Riemannian metric $\tilde{g}_{\infty}$ on $D$ and that there exists a point $x \in D$ such that the sequence of the $1$-jets $j^{1}\tilde{f}_{n}(x)$ converges. If the sequence $\tilde{f}_{n}$ does not converge in the $C^{\infty}$ topology in a neighbourhood of $x$, then there exists a maximal geodesic $\gamma$ of $(D, \tilde{g}_{\infty})$ and a geodesic arc $\Gamma$ of $(X, \tilde{g})$ such that the sequence $(\tilde{f}_{n})_{|_{\gamma}}$ converges towards an isometry $\tilde{f}_{\infty}: \gamma \rightarrow \Gamma$.
\end{teo}

\footnotetext[1]{We recall that a sequence of isometric immersions is said to be uniformly elliptic if the corresponding shape operators have uniformly positive determinant.} 

 We start with a straightforward application of the Maximum Principle, which we recall here in the form useful for our purposes (see e.g. Prop. 4.6 in \cite{folKsurfaces}).

\begin{prop}[Maximum Principle] Let $\Sigma_{1}$ and $\Sigma_{2}$ two future-convex space-like surfaces embedded in a GHMC $AdS_{3}$ manifold $M$. If they intersect in a point $x$ and $\Sigma_{1}$ is in the future of $\Sigma_{2}$ then the product of the principal curvatures of $\Sigma_{2}$ is smaller than the product of the principal curvatures of $\Sigma_{1}$.
\end{prop}

\begin{prop}Let $\Sigma$ be a future-convex space-like surface embedded into a GHMC $AdS_{3}$ manifold $M$. Suppose that the Gaussian curvature of $\Sigma$ is bounded between $-\infty < \K_{min}\leq \K_{max} < -1$. Denote with $S_{min}$ and $S_{max}$ the unique future-convex space-like surfaces with constant curvature $\K_{min}$ and $\K_{max}$ embedded in $M$ (Corollary 4.7 in \cite{folKsurfaces}). Then $\Sigma$ is in the past of $S_{max}$ and in the future of $S_{min}$.
\end{prop}
\begin{proof} Consider the unique (Corollary 4.7 in \cite{folKsurfaces}) $\K$-time 
\[
	T:I^{-}(\partial_{-}C(M))\rightarrow (-\infty,-1)\ ,
\]
i.e. the unique function defined on the past of the convex core of $M$ such that the level sets $T^{-1}(\K)$ are future-convex space-like surfaces of constant curvature $\K$. The restriction of $T$ to $\Sigma$ has a maximum $t_{max}$ and a minimum $t_{min}$. Consider the level sets $L_{min}=T^{-1}(t_{min})$ and $L_{max}=T^{-1}(t_{max})$. By construction $\Sigma$ is in the future of $L_{min}$ and they intersect in a point $x$, hence, by the Maximum Principle and the Gauss equation, we obtain the following inequality for the Gaussian curvature of $\Sigma$ at the point $x$:
\[
	t_{min}\geq \K(x) \geq \K_{min} \ .
\]
Similarly we obtain that $t_{max}\leq \K(y)\leq \K_{max}$, where $y$ is the point of intersection between $L_{max}$ and $\Sigma$. But this implies that $\Sigma$ is in the past of the level set $T^{-1}(\K_{max})$ and in the future of the level set $T^{-1}(\K_{min})$, which correspond respectively to the surfaces $S_{max}$ and $S_{min}$ by uniqueness.
\end{proof}

\begin{cor}\label{convergenza}Let $g_{n}$ be a compact family of metrics in the $C^{\infty}$ topology with curvatures $\K < -1$ on a surface $S$. Let $f_{n}: (S, g_{n})\rightarrow M_{n}=(S\times \R, h_{n})$ be a sequence of isometric embeddings of $(S, g_{n})$ as future-convex space-like surfaces into GHMC $AdS_{3}$ manifolds. If the sequence $h_{n}$ converges to an $AdS$ metric $h_{\infty}$ in the $C^{\infty}$-topology, then $f_{n}$ converges $C^{\infty}$, up to subsequences, to an isometric embedding into $M_{\infty}=(S\times \R, h_{\infty})$.
\end{cor}
\begin{proof} Consider the equivariant isometric embeddings $\tilde{f}_{n}: (\tilde{S}, \tilde{g}_{n}) \rightarrow \tilde{AdS}_{3}$ obtained by lifting $f_{n}$ to the universal cover. We denote with $\tilde{S}_{n}$ the images of the disc $\tilde{S}$ under the map $\tilde{f}_{n}$ and let $\tilde{h}_{n}$ be the lift of the Lorentzian metrics $h_{n}$ on $\tilde{AdS}_{3}$. By hypothesis $\tilde{f}_{n}^{*}\tilde{h}_{n}=\tilde{g}_{n}$ admits a subsequence converging to $\tilde{g}_{\infty}$. \\
\indent Fix a point $x \in \tilde{S}$. Since the isometry group of $\tilde{AdS}_{3}$ acts transitively on points and frames, we can suppose that $\tilde{f}_{n}(x)=y \in \tilde{AdS}_{3}$ and $j^{1}\tilde{f}_{n}(x)=z$ for every $n \in \mathbb{N}$. \\
\indent Moreover, the condition on the curvature of the metrics $g_{n}$ guarantees that the sequence $\tilde{f}_{n}$ is uniformly elliptic. \\
Therefore, we are under the hypothesis of Theorem \ref{degenere}.\\
\indent The previous proposition allows us to determine precisely in which region of $M_{n}$ each surface $f_{n}(S)$ lies.
Since the family of metrics $g_{n}$ is compact, the curvatures $\K_{n}$ of the surfaces $f_{n}(S)$ in $M_{n}$ are uniformly bounded $\K^{min} \leq \K_{n} \leq \K^{max}\leq -1-3\epsilon$ for some $\epsilon>0$. By the previous proposition each surface $f_{n}(S)$ is in the past of $\Sigma_{n}^{max}$ and in the future of $\Sigma_{n}^{min}$, where $\Sigma_{n}^{min}$ and $\Sigma_{n}^{max}$ are the unique future-convex space-like surfaces of $M_{n}$ with constant curvature $\K^{min}$ and $\K^{max}$. Let $\Sigma_{\epsilon}$ be the unique future-convex space-like surface in $M_{\infty}$ with constant curvature $-1-2\epsilon$. We think of $\Sigma$ as a fixed surface embedded in $S\times \R$ and we change the Lorentzian metric of the ambient space. Since $h_{n}$ converges to $h_{\infty}$, the metrics induced on $\Sigma_{\epsilon}$ by $h_{n}$ converge to the metric induced on $\Sigma_{\epsilon}$ by $h_{\infty}$. In particular, for $n$ sufficiently large the curvature of $\Sigma_{\epsilon}$ as surface embedded in $M_{n}=(S \times \R, h_{n})$ is bounded between $-1-3\epsilon$ and $-1-\epsilon$. Therefore, $\Sigma_{\epsilon}$ is convex in $M_{n}$ and by the previous proposition $\Sigma_{\epsilon}$ is in the future of $\Sigma_{n}^{max}$ for every $n$ sufficiently big. This implies that each surface $f_{n}(S)$ is in the past of the surface $\Sigma_{\epsilon}$. \\
\indent We can now conclude that the sequence $f_{n}$ must converge to an isometric embedding. Suppose by contradiction that the sequence $\tilde{f}_{n}$ is not convergent in the $C^{\infty}$ topology in a neighbourhood of $x$, then there exists a maximal geodesic $\tilde{\gamma}$ of $(\tilde{S}, \tilde{g}_{\infty})$ and a geodesic segment $\tilde{\Gamma}$ in $\tilde{AdS}_{3}$ such that $(\tilde{f}_{n})_{|_{\tilde{\gamma}}}$ converges to an isometry $\tilde{f}_{\infty}: \tilde{\gamma} \rightarrow \tilde{\Gamma}$. This implies that $\tilde{\Gamma}$ has infinite length. The projection of $\tilde{\Gamma}$ must be contained in the past of $\Sigma_{\epsilon}$, because each $f_{n}(S)$ is contained there for $n$ sufficiently large. But the past of $\Sigma_{\epsilon}$ is disjoint from the convex core of $M_{\infty}$ and this contradicts the following lemma.
\end{proof}

\begin{lemma}\label{geodetiche} In a GHMC $AdS_{3}$-manifold every complete space-like geodesic is contained in the convex core.
\end{lemma}
\begin{proof} Let $c$ be a complete space-like geodesic in a GHMC $AdS_{3}$ manifold $M$. By a result of Mess (\cite{Mess}), we can realise $M$ as the quotient of the domain of dependence $D(\rho)\subset AdS_{3}$ of a curve $\rho$ on the boundary at infinity by the action of the fundamental group of $S$. The lift $\bar{c}$ of $c$ has ending points on the curve $\rho$, hence $\bar{c}$ is contained in the convex hull of $\rho$ into $AdS_{3}$ and its projection is contained in the convex core of $M$. 
\end{proof}

\begin{oss}Clearly, the same result holds for equivariant isometric embeddings of past-convex space-like surfaces, as it is sufficient to reverse the time-orientation.
\end{oss}

\begin{cor}\label{properdef}The functions $\Phi^{\pm}:W^{\pm} \rightarrow \dT$ are proper.
\end{cor}
\begin{proof}We prove the claim for the function $\Phi^{-}$, the other case being analogous. Let $(h_{l}(g_{t_{n}}, b_{t_{n}}), h_{r}(g_{t_{n}}, b_{t_{n}}))\in \dT$ be a convergent sequence in the image of the map $\Phi^{-}$. This means that the sequence of GHMC $AdS_{3}$ manifolds $M_{n}$ parametrised by $(h_{l}(g_{t_{n}}, b_{t_{n}}), h_{r}(g_{t_{n}}, b_{t_{n}}))$ is convergent. By definition of the map $\Phi^{-}$, each $M_{n}$ contains an embedded future-convex, space-like surface isometric to $(S, g_{t_{n}})$, whose immersion $f_{n}$ into $M_{n}$ is represented by the Codazzi operator $b_{t_{n}}$. By Corollary \ref{convergenza}, the sequence of isometric immersions $f_{n}$ is convergent up to subsequences, thus $\Phi^{-}$ is proper. 
\end{proof}

 This allows us to show that for every metric $g_{-}$ and for every smooth path of metrics $\{g_{t}^{+}\}_{t \in [0,1]}$ on $S$ with curvature $\K <-1$ the intersection between $\Phi^{+}(W^{+})$ and $\phi^{-}_{g_{-}}(I(S, g_{-})^{-})$ is compact. This will follow combining some technical results about the geometry of $AdS_{3}$ manifolds and length-spectrum comparisons.
 
\begin{defi} Let $g$ be a metric with negative curvature on $S$. We define the length function
\[
	\ell_{g}: \pi_{1}(S) \rightarrow \R^{+} 
\]
which associates to every homotopy non-trivial loop on $S$, the length of its $g$-geodesic representative.
\end{defi}

We recall that when $g$ is a hyperbolic metric, Thurston proved (see e.g. \cite{lavoriThurston}) that the length function can be extended uniquely to a function on the space of measured geodesic laminations on $S$, which we still denote with $\ell_{g}$.\\

We will need the following technical results: 

\begin{lemma}[Lemma 9.6 in \cite{landslide2}]\label{stimametriche} Let $N$ be a globally hyperbolic compact $AdS_{3}$-manifold foliated by future-convex space-like surfaces. Then, the sequence of metrics induced on each surface decreases when moving towards the past. In particular, if $\Sigma_{1}$ and $\Sigma_{2}$ are two future-convex space-like surfaces with $\Sigma_{1}$ in the future of $\Sigma_{2}$, then for every closed geodesic $\gamma$ in $\Sigma_{1}$ we have
\[
	\ell_{g_{2}}(\gamma')\leq \ell_{g_{1}}(\gamma) \ ,
\]
where $\gamma'$ is the closed geodesic on $\Sigma_{2}$ homotopic to $\gamma$ and $g_{1}$ and $g_{2}$ are the induced metric on $\Sigma_{1}$ and $\Sigma_{2}$, respectively. 
\end{lemma}

\begin{lemma}\label{compactmetric} Let $g_{n}$ be a compact family of smooth metrics on $S$ with curvature less than $-1$. Let $m_{n}$ be a family of hyperbolic metrics such that
\[
	\ell_{g_{n}}(\gamma) \leq \ell_{m_{n}}(\gamma)
\]
for every $\gamma \in \pi_{1}(S)$. Then $m_{n}$ lies in a compact subset of the Teichm\"uller space of $S$.
\end{lemma}
\begin{proof} The idea is to use Thurston asymmetric metric on Teichm\"uller space. To this aim, we will deduce from the hypothesis a comparison between the length spectrum of $m_{n}$ and that of the hyperbolic metrics $h_{n}$ in the conformal class of $g_{n}$. \\
\indent Let $\K<-1$ be the infimum of the curvatures of the family $g_{n}$. Since $g_{n}$ is a compact family, $\K> -\infty$. Let $\bar{g}_{n}=-\frac{1}{\K}h_{n}$ be the metrics of constant curvature $\K$ in the conformal class of $g_{n}$. We claim that 
\[
	\ell_{h_{n}}(\gamma) \leq \sqrt{|\K|}\ell_{m_{n}}(\gamma)
\]
for every $\gamma \in \pi_{1}(S)$. For instance, if we write $\bar{g}_{n}=e^{2u_{n}}g_{n}$, the smooth function $u_{n}: S \rightarrow \R$ satisfies the differential equation
\[
	e^{2u_{n}(x)}\K= \K_{g_{n}}(x)+ \Delta_{g_{n}}u_{n}(x) \ ,
\]
where $\K_{g_{n}}$ is the curvature of $g_{n}$. Since $\K_{g_{n}}\geq \K$, $\Delta_{g_{n}}u_{n}$ is positive at the point of maximum of $u_{n}$ and $\K<-1$, we deduce that $e^{2u_{n}}\leq 1$, hence
\[
	\ell_{\bar{g}_{n}}(\gamma)\leq \ell_{g_{n}}(\gamma)
\]
for every $\gamma \in \pi_{1}(S)$. It is then clear that
\[
	\ell_{\bar{g}_{n}}(\gamma)=\frac{1}{\sqrt{\K}}\ell_{h_{n}}(\gamma)
\]
for every $\gamma \in \pi_{1}(S)$ and the claim follows. \\
\indent Moreover, by the inequality
\[
	\ell_{h_{n}}(\gamma) \leq \sqrt{|\K|} \ell_{g_{n}}(\gamma) \ \ \ \  \forall \ \gamma \in \pi_{1}(S)
\]
we deduce that $h_{n}$ is contained in a compact set of Teichm\"uller space: if that were not the case, there would exists a curve $\gamma$ such that $\ell_{h_{n}}(\gamma)\xrightarrow{n \to \infty} +\infty$, which is impossible because $g_{n}$ is a compact family. \\
\indent We can conclude now using Thurston asymmetric metric: given two hyperbolic metrics $h$ and $h'$, Thurston asymmetric distance between $h$ and $h'$ is defined as
\[
	d_{Th}(h, h')=\sup_{\gamma \in \pi_{1}(S)}\log\left(\frac{\ell_{h}(\gamma)}{\ell_{h'}(\gamma)}\right) \ .
\]
It is well-known (\cite{Thurston}) that if $h'_{n}$ is a divergent sequence than $d_{Th}(K, h_{n}) \to + \infty$, where $K$ is any compact set in Teichm\"uller space. Now, by the length spectrum comparison 
\[
	 \ell_{h_{n}}(\gamma) \leq \sqrt{|\K|}\ell_{m_{n}}(\gamma)   \ \ \ \  \forall \ \gamma \in \pi_{1}(S) \ ,
\]
we deduce that $d_{Th}(h_{n}, m_{n}) \leq \log(\sqrt{|\K|})< +\infty$, hence $m_{n}$ must be contained in a compact set. 
\end{proof}

We will need also the following fact about the geometry of the convex core of a GHMC $AdS_{3}$ manifold.

\begin{lemma}[Prop. 5 in \cite{Diallo}] \label{stimalaminazioni} Let $M$ be a GHMC $AdS_{3}$ manifold. Denote by $m^{+}$ and $m^{-}$ the hyperbolic metrics on the upper and lower boundary of the convex core of $M$. Let $\lambda^{+}$ and $\lambda^{-}$ be the measured geodesic laminations on the upper and lower boundary of the convex core of $M$. For all $\epsilon>0$, there exists some $A>0$ such that, if $m^{+}$ is contained in a compact set and $\ell_{m^{+}}(\lambda^{+}) \geq A$, then $\ell_{m^{-}}(\lambda^{+})\leq \epsilon \ell_{m^{+}}(\lambda^{+})$.
\end{lemma}

\begin{prop}\label{compact}For every metric $g^{-}$ and for every smooth path of metrics $\{g_{t}^{+}\}_{t \in [0,1]}$ on $S$ with curvature $\K < -1$, the set $(\Phi^{+}\times \phi^{-}_{g^{-}})^{-1}(\Delta)$ is compact. 
\end{prop}
\begin{proof}We need to prove that every sequence of isometric embeddings $(b^{+}_{t_{n}}, b^{-}_{n})$ in $(\Phi^{+} \times \phi^{-}_{g^{-}})^{-1}(\Delta)$ admits a convergent subsequence. By definition, for every $n \in \mathbb{N}$, there exists a GHMC $AdS_{3}$ manifold $M_{n}$ containing a past-convex surface isometric to $(S, g^{+}_{t_{n}})$ with shape operator $b^{+}_{t_{n}}$ and a future-convex surface isometric to $(S, g^{-})$ with shape operator $b^{-}_{n}$. By Lemma \ref{stimametriche} and Lemma \ref{compactmetric}, the metrics $m_{n}^{+}$ and $m_{n}^{-}$ on the upper and lower boundary of the convex core of $M_{n}$ are contained in a compact set of $\T(S)$ . \\ 
\indent We are going to prove now that the sequences of left and right metrics of $M_{n}$ are contained in a compact set of Teichm\"uller space, as well. Suppose by contradiction that the sequence of left metric $h_{l_{n}}$ of $M_{n}$ is not contained in a compact set. By Mess parameterisation (see Section \ref{parametrization}, or \cite{Mess}), the left metrics are related to the metrics $m_{n}^{+}$ and to the measured geodesic laminations $\lambda_{n}^{+}$ of the upper-boundary of the convex core by an earthquake: 
\[
	h_{l_{n}}=E_{\lambda_{n}^{+}}^{l}(m_{n}^{+}) \ .
\]
Since $h_{l_{n}}$ is divergent, the sequence of measured laminations $\lambda_{n}^{+}$ is divergent, as well. In particular, this implies that $\ell_{m_{n}^{+}}(\lambda_{n}^{+})$ goes to infinity. Therefore, by Lemma \ref{stimalaminazioni}, for every $\epsilon>0$ there exists $n_{0}$ such that the inequality $\ell_{m_{n}^{-}}(\lambda^{+}_{n})\leq \epsilon\ell_{m_{n}^{+}}(\lambda^{+}_{n})$ holds for $n\geq n_{0}$. From this we deduce a contradiction, because we prove that the inequality 
\[
		\ell_{m_{n}^{-}}(\lambda^{+}_{n})\leq \epsilon\ell_{m_{n}^{+}}(\lambda^{+}_{n})  \ \ \ \ \ \forall \ n\geq n_{0}
\]
implies that the sequence $m_{n}^{-}$ is divergent, which contradicts what we proved in the previous paragraph. For instance, if $m_{n}^{-}$ were contained in a compact set of Teichm\"uller space, there would exist (using again Thurston's asymmetric metric) a constant $C>1$ such that
\[
	\frac{\ell_{m_{n}^{+}}(\gamma)}{\ell_{m_{n}^-}(\gamma)} \leq C  \ \ \ \ \forall \ n\geq n_{0} \ .
\]
By density this inequality must hold also for every measured geodesic lamination on $S$. But we have seen that for every $\epsilon >0$ we can find $n_{0}$ such that for every $n \geq n_{0}$ we have
\[
	\frac{\ell_{m_{n}^{+}}(\lambda_{n}^{+})}{\ell_{m_{n}^-}(\lambda_{n}^{+})} \geq \frac{1}{\epsilon} \ ,
\]
thus obtaining a contradiction. \\
A similar argument proves that also the sequence of right metrics $h_{r_{n}}$ must be contained in a compact set of $\T(S)$. \\
\indent Since the sequences of left and right metrics of $M_{n}$ converge, up to subsequence, we can concretely realise the corresponding subsequence $M_{n}$ as $(S \times \R, h_{n})$ such that $h_{n}$ converges in the $C^{\infty}$-topology to an Anti-de Sitter metric $h_{\infty}$ and each $M_{n}$ contains a future-convex space-like surface with embedding data $(g^{-}, b^{-}_{n})$ and a past-convex space-like surface with embedding data $(g^{+}_{t_{n}}, b_{t_{n}}^{+})$. The proof is then completed applying Corollary \ref{convergenza}. 
\end{proof}
    
\section{Prescription of an isometric embedding and half-holonomy representation}\label{halfholonomy}
This section is dedicated to the proof of the following result about the existence of an $AdS_{3}$ manifold with prescribed left metric containing a convex space-like surface with prescribed induced metric:

\begin{prop}\label{halfhol} Let $g$ be a metric on $S$ with curvature less than $-1$ and let $h$ be a hyperbolic metric on $S$. There exists a GHMC $AdS_{3}$ manifold $M$ with left metric isotopic to $h$ containing a past-convex space-like surface isometric to $(S,g)$. 
\end{prop}

If we denote with  
\[
	p_{1}: \dT \rightarrow \T(S)
\]
the projection onto the left factor, Propostition \ref{halfhol} is equivalent to proving that the map $p_{1}\circ\phi^{+}_{g}: I(S,g)^{+}\rightarrow \T(S)$ is surjective. After showing that $p_{1}\circ\phi^{+}_{g}$ is proper (Corollary \ref{properdefproj}), this will follow from the fact that its degree (mod $2$) is non-zero.\\
\\
\indent In order to prove properness of the map $p_{1}\circ\phi^{+}_{g}$, we will need the following well-known result about the behaviour of the length function while performing an earthquake. 
\begin{lemma}[Lemma 7.1 in \cite{terremoti}]\label{stimaterremoti}Given a geodesic lamination $\lambda\in \mathpzc{M}\mathpzc{L}(S)$ and a hyperbolic metric $g\in \T(S)$, let $g'=E_{l}^{\lambda}(g)$. Then for every closed geodesic $\gamma$ in $S$ the following estimate holds
\[
	\ell_{g}(\gamma)+\ell_{g'}(\gamma)\geq \lambda(\gamma) \ .
\]
\end{lemma}

\begin{prop}\label{projproper}For every path of metrics $\{g_{t}\}_{t\in [0,1]}$ with curvature less than $-1$, the projection $p_{1}: \Phi^{+}(W^{+}) \rightarrow \T(S)$ is proper.
\end{prop}
\begin{proof}Let $h_{l}(g_{t_{n}}, b_{t_{n}})$ be a convergent sequence of left metrics. We need to prove that the corresponding sequence of right metrics $h_{r}(g_{t_{n}}, b_{t_{n}})$ is convergent, as well. By hypothesis, $(S,g_{t_{n}})$ is isometrically embedded as past-convex space-like surface in each GHMC $AdS_{3}$ manifold $M_{n}$ parametrised by $(h_{l}(g_{t_{n}}, b_{t_{n}}), h_{r}(g_{t_{n}}, b_{t_{n}}))$.  By Lemma \ref{stimametriche} and Lemma \ref{compactmetric}, the metrics $m_{n}^{+}$ on the past-convex boundary of the convex core of $M_{n}$ are contained in a compact set of $\T(S)$. Moreover, by a result of Mess (\cite{Mess}), the left metrics $h_{l}(g_{t_{n}}, b_{t_{n}})$, the metrics $m_{n}^{+}$ and the measured laminations on the convex core $\lambda^{+}_{n}$ are related by an earthquake
\[
	h_{l}(g_{t_{n}}, b_{t_{n}})=E_{l}^{\lambda_{n}^{+}}(m_{n}^{+}) \ .
\]
Since $h_{l}(g_{t_{n}}, b_{t_{n}})$ is convergent, by Lemma \ref{stimaterremoti}, the sequence of measured laminations $\lambda_{n}^{+}$ must be contained in a compact set. Therefore, by continuity of the right earthquake
	\begin{align*}
	E_{r}: \T(S) &\times \mathpzc{M}\mathpzc{L}(S) \rightarrow \T(S)\\
		(h,\lambda) &\mapsto E_{r}^{\lambda}(h)
	\end{align*}
the sequence 
\[
	h_{r}(g_{t_{n}}, b_{t_{n}})=E_{r}^{\lambda_{n}^{+}}(m_{n}^{+})
\]
is convergent, up to subsequences. 
\end{proof}

 In particular, considering a constant path of metrics, we obtain the following:
\begin{cor}\label{properdefproj}The projection $p_{1}:\phi^{+}_{g}(I(S,g)^{+}) \rightarrow \T(S)$ is proper .
\end{cor}
 
\begin{prop}\label{degree}For every metric $g$ of curvature $\K<-1$, the map 
\[
	p_{1}\circ\phi^{+}_{g}: I(S,g)^{+}\rightarrow \T(S)
\]
is proper of degree $1$ mod $2$.
\end{prop}
\begin{proof}Consider a path of metrics $(g_{t})_{t\in [0,1]}$ with curvature less than $-1$ connecting $g=g_{0}$ with a metric of constant curvature $g_{1}$. By Corollary \ref{properdef} and Corollary \ref{properdefproj}, the maps $p_{1}\circ\phi_{g_{0}}^{+}:I(S,g_{0})^{+}\rightarrow \T(S)$ and $p_{1}\circ\phi^{+}_{g_{1}}:I(S,g_{1})^{+}\rightarrow \T(S)$ are proper and cobordant, hence they have the same degree (mod $2$). (This follows from Remark \ref{defdegree}, Proposition \ref{invcobgen} and Proposition \ref{projproper}). Thus, we can suppose that $g$ has constant curvature $\K<-1$. \\
\indent We notice that there exists a unique element in $I(S,g)^{+}$ such that $h_{l}(g, b)=-\K g$: a direct computation shows that $b=\sqrt{-\K-1}E$ works and uniqueness follows by the theory of landslides developed in \cite{landslide1}. We sketch here the argument and we invite the interested reader to consult the aforementioned paper for more details. Pick $\theta \in (0,\pi)$ such that $\K=-\frac{1}{\cos^{2}(\theta/2)}$. The landslide
	\begin{align*}
	L_{e^{i\theta}}^{1}:\T(S)\times \T(S) &\rightarrow \T(S)\\
		(h,h^{*}) &\mapsto h'
	\end{align*}
associates to a couple of hyperbolic metrics $(h,h^{*})$, the left metric of a GHMC $AdS_{3}$ manifold containing a space-like embedded surface with induced first fundamental form $I=\cos^{2}(\theta/2)h$ and third fundamental form $III=\sin^{2}(\theta/2)h^{*}$. It has been proved (Theorem 1.14 in \cite{landslide1}) that for every $(h,h')\in \dT$, there exists a unique $h^{*}$ such that $L_{e^{i\theta}}^{1}(h,h^{*})=h'$. Moreover, the shape operator $b$ of the embedded surface can be recovered by the formula (Lemma 1.9 in \cite{landslide1})
\[
	b=\tan(\theta/2)B
\]
where $B:TS \rightarrow TS$ is the unique $h$-self-adjoint operator such that $h^{*}=h(B\cdot, B\cdot)$. Therefore, if we choose $h=h'=-\K g$, the uniqueness of the operator $b$ follows by the uniqueness of $h^{*}$ and $B$. \\
\indent Hence, the degree (mod $2$) of the map is $1$, provided $-\K g$ is a regular value. Let $\dot{b}\in T_{b}I(S,g)^{+}$ be a non-trivial tangent vector. We remark that, since elements of $I(S,g)^{+}$ are $g$-self-adjoint, Codazzi tensor of determinant $-1-\K$, the tangent space $T_{b}I(S,g)^{+}$ can be identified with the space of traceless, Codazzi, $g$-self-adjoint tensors. We are going to prove that the deformation induced on the left metric is non-trivial, as well. Let $b_{t}$ be a path in $I(S,g)^{+}$ such that $b_{0}=b=\sqrt{-\K-1}E$ and $\frac{d}{dt}b_{t}=\dot{b}$ at $t=0$. The complex structures induced on $S$ by the metrics $h_{l}(g, b_{t})$ are 
\[
	J_{t}=(E+Jb_{t})^{-1}J(E+Jb_{t})
\]
where $J$ is the complex structure induced by $g$. Taking the derivative of this expression at $t=0$ we get
\[
	\dot{J}=\frac{2}{\K}[E-\sqrt{-\K-1}J]\dot{b}
\]
which is non-trivial in $T_{-\K g}\T(S)$ because, as explained in Theorem 1.2 of \cite{Tromba}, the space of traceless and Codazzi operators in $T_{J}\mathpzc{A}$ has trivial intersection with the kernel of the differential of the projection $\pi: \mathpzc{A} \rightarrow \T(S)$, which sends a complex structure $J$ to its isotopy class. 
\end{proof}

In particular, for every smooth metric $g$ on $S$ with curvature less than $-1$, the map $p_{1}\circ\phi^{+}_{g}: I(S,g)^{+}\rightarrow \T(S)$ is surjective (a proper, non-surjective map has vanishing degree (mod $2$)) and we deduce Proposition \ref{halfhol}.

\section{Proof of the main result}\label{proofmainthm} 
 We have now all the ingredients to prove Theorem \ref{mainteo}. As outlined in the Introduction, the first step consists of verifying that in one particular case, i.e. when we choose the metrics $g'_{+}=-\frac{1}{\K}h$ and $g'_{-}=-\frac{1}{\K^{*}}h$, where $h$ is any hyperbolic metric and $\K^{*}=-\frac{\K}{\K+1}=\K=-2$, the maps $\phi^{+}_{g'_{+}}$ and $\phi^{-}_{g'_{-}}$ have a unique transverse intersection. \\
\indent It is a standard computation to verify that $b^{+}=E$ and $b^{-}=-E$ are Codazzi operators corresponding to an isometric embedding of $(S, g'_{+})$ as a past-convex space-like surface and to an isometric embedding of $(S, g'_{-})$ as a future-convex space-like surface respectively into the GHMC $AdS_{3}$ manifold $M$ parametrised by $(h,h) \in \dT$. This manifold $M$ is unique due to the following:

\begin{teo}[Theorem 1.15 in \cite{landslide2}]Let $h_{+}$ and $h'_{-}$ be hyperbolic metrics and let $\K_{+}$ and $\K_{-}$ be real numbers less than $-1$. There exists a GHMC $AdS_{3}$ manifold $M$ which contains an embedded future-convex space-like surface with induced metric $\frac{1}{|\K_{-}|}h_{-}$ and an embedded past-convex space-like surface with induced metric $\frac{1}{|\K_{+}|}h_{+}$. Moreover, if $\K_{+}=-\frac{\K_{-}}{\K_{-}+1}$, then $M$ is unique. 
\end{teo}

We notice that $M$ is Fuchsian, i.e. it is parametrised by a couple of isotopic metrics in Teichm\"uller space. A priori, there might be other isometric embeddings of $(S, g'_{+})$ as a past-convex space-like surface and of $(S, g'_{-})$ as a future-convex space-like surface into $M$ not equivalent to the ones found before. Actually, this is not the case due to the following result about isometric embeddings of convex surfaces into Fuchsian Lorentzian manifolds:

\begin{teo}[Theorem 1.1 in \cite{casofuchsiano}]\label{unicità} Let $(S,g)$ be a Riemannian surface of genus $\tau\geq 2$ with curvature strictly smaller than $-1$. Let $x_{0} \in \tilde{AdS}_{3}$ be a fixed point. There exists an equivariant isometric embedding $(f, \rho)$ of $(S,g)$ into $\tilde{AdS}_{3}$ such that $\rho$ is a representation of the fundamental group of $S$ into the group $\Isom(\tilde{AdS}_{3}, x_{0})$ of isometries of $\tilde{AdS}_{3}$ fixing $x_{0}$. Such an embedding is unique modulo $\Isom(\tilde{AdS}_{3}, x_{0})$.
\end{teo}

 As a consequence, if we denote with $\Delta$ the diagonal of $\T(S)^{2}\times \T(S)^{2}$, we have proved that
\[
 (\phi^{+}_{g'_{+}}\times \phi^{-}_{g'_{-}})^{-1}(\Delta)=(E, -E)\in I(S,g'_{+})^{+}\times I(S,g'_{-})^{-} \ .
\]

We need to verify next that at this point the intersection 
\[
	\phi^{+}_{g'_{+}}(I(S,g'_{+})^{+})\cap \phi^{-}_{g'_{-}}(I(S,g'_{-})^{-})
\]
is transverse. Suppose by contradiction that the intersection is not transverse, then there exists a non-trivial tangent vector $\dot{b}^{+} \in T_{E}I(S,g'_{+})^{+}$ and a non-trivial tangent vector $\dot{b}^{-} \in T_{-E}I(S,g'_{-})^{-}$ such that 
\[
	d\phi^{+}_{g'_{+}}(\dot{b}^{+})=d\phi^{-}_{g'_{-}}(\dot{b}^{-})\in T_{h}\T(S) \times T_{h}\T(S) \ .
\]
We recall that elements of $T_{E}I(S,g'_{+})^{+}$ can be represented by traceless, $g'_{+}$-self-adjoint, Codazzi operators. With this in mind, let us compute explicitly $d\phi^{+}_{g'_{+}}(\dot{b}^{+})$. Let $b_{t}^{+}$ be a smooth path in $I(S,g'_{+})^{+}$ such that $b^{+}_{0}=E$ and $\frac{d}{dt}_{|_{t=0}}b^{+}_{t}=\dot{b}^{+}\ne 0$. The complex structures induced on $S$ by the left metrics $h_{l}(b_{t})$ are
\[
	J_{l}^{+}=(E+Jb_{t}^{+})^{-1}J(E+Jb_{t}^{+}) \ ,
\]
where $J$ is the complex structure of $(S, g'_{+})$. We compute now the derivative of this expression at $t=0$. First notice that, since the operators $b_{t}$ are $g_{+}'$-self-adjoint, $Jb_{t}^{+}$ is traceless, hence the Hamilton-Cayley equation reduces to $(Jb_{t}^{+})^{2}+\det(Jb_{t}^{+})E=(Jb_{t}^{+})^{2}+E=0$. We deduce that
\[
	(E+Jb_{t}^{+})(E-Jb_{t}^{+})=2E \ .
\]
Therefore, the variation of the complex structures induced by the left metrics is 
	\begin{align*}
	\dot{J}_{l}^{+}&=\frac{d}{dt}_{|_{t=0}}J_{l}^{+}=\frac{d}{dt}_{|_{t=0}}\frac{1}{2}(E-Jb_{t}^{+})J(E+Jb_{t}^{+})\\
		&=\frac{1}{2}(-J\dot{b}^{+})J(E+J)+\frac{1}{2}(E-J)J^{2}\dot{b}^{+}\\
		&=-(E-J)\dot{b}^{+}
	\end{align*}
where, in the last passage we used the fact that, since $\dot{b}^{+}$ is traceless and symmetric, the relation $J\dot{b}^{+}=-\dot{b}^{+}J$ holds. \\
With a similar procedure we compute the variation of the complex structures of the right metrics and we obtain
\[
	\dot{J}_{r}^{+}=(E+J)\dot{b}^{+} \in T_{J}\mathpzc{A} \ .
\]
Noticing that $\dot{J}_{l}^{+}$ and $\dot{J}_{r}^{+}$ are both traceless Codazzi operators, the image of $\dot{b}^{+}$ under the differential $d\phi^{+}_{g'_{+}}$ is simply
\[
	d\phi^{+}_{g'_{+}}(\dot{b}^{+})=(-(E-J)\dot{b}^{+}, (E+J)\dot{b}^{+})\in T_{h}\T(S)\times T_{h}\T(S)
\]
because, as explained in Theorem 1.2 of \cite{Tromba}, the space of traceless and Codazzi operators in $T_{J}\mathpzc{A}$ is in direct sum with the kernel of the differential of the projection $\pi: \mathpzc{A} \rightarrow \T(S)$, which sends a complex structure $J$ to its isotopy class and gives an isomorphism between the space of traceless, Codazzi, self-adjoint tensors and $T_{h}\T(S)$. \\
With a similar reasoning we obtain that 
\[
	d\phi^{-}_{g'_{-}}(\dot{b}^{-})=(-(E+J)\dot{b}^{-}, (E-J)\dot{b}^{-}) \in  T_{h}\T(S)\times T_{h}\T(S) \ .
\]
By imposing that $d\phi^{+}_{g'_{+}}(\dot{b}^{+})=d\phi^{-}_{g'_{-}}(\dot{b}^{-})$ we obtain the linear system
\[
	\begin{cases}
	(-E+J)\dot{b}^{+}=-(E+J)\dot{b}^{-} \\
	(E+J)\dot{b}^{+}=(E-J)\dot{b}^{-}
	\end{cases}
\]
which has solutions if and only if $\dot{b}^{+}=\dot{b}^{-}=0$. Therefore, the intersection is transverse and we can finally state that 
\[
	\Im(\phi^{+}_{g'_{+}}, \phi_{g'_{-}}^{-})=1 \ .
\]
\\
\indent Now we use the theory described in Section \ref{intersectiontheory} to prove that an intersection persists under a deformation of one metric that fixes the other. Let $g_{+}$ and $g_{-}$ be two arbitrary metrics on $S$ with curvature less than $-1$. We will still denote with $g'_{+}$ and with $g'_{-}$ the metrics introduced in the previous paragraph with self-dual constant curvature and in the same conformal class. Consider two paths of metrics $\{g_{+}^{t}\}_{t \in [0,1]}$ and $\{g_{-}^{t}\}_{t \in [0,1]}$ with curvature less than $-1$ such that $g_{+}^{0}=g_{+}$, $g_{+}^{1}=g'_{+}$, $g_{-}^{0}=g_{-}$ and $g_{-}^{1}=g'_{-}$. 
We will first prove that 
\[
	\phi_{g_{+}}^{+}(I(S,g_{+})^{+}) \cap \phi^{-}_{g'_{-}}(I(S,g'_{-})^{-}) \ne \emptyset \ .
\]
Suppose by contradiction that this intersection is empty. Then the map $\phi_{g_{+}}^{+} \times \phi_{g'_{-}}^{-}$ is trivially transverse to $\Delta$. Consider the manifold
\[
	W^{+}=\bigcup_{t \in [0,1]} I(S, g^{+}_{t})^{+}
\]
and the map 
\[
	\Phi^{+}\times \phi_{g'_{-}}^{-}: X=W^{+} \times I(S, g'_{-})^{-} \rightarrow (\T(S))^{4}=Y
\]
as defined in Section \ref{studymaps}. By assumption the restriction of $\Phi^{+}\times \phi_{g'_{-}}^{-}$ to the boundary is transverse to $\Delta$ and by Proposition \ref{compact}, the set $D=(\Phi^{+}\times  \phi_{g'_{-}}^{-})^{-1}(\Delta)$ is compact. Let $B$ be the interior of a compact set containing $D$ and let $C=(X \setminus B)\cup \partial X$. By construction, $\Phi^{+}\times \phi_{g_{-}}^{-}$ is transverse to $\Delta$ along the closed set $C$. Applying Theorem \ref{perturb}, there exists a smooth map $\Psi: X \rightarrow Y$ which is transverse to $\Delta$ and which coincides with $\Phi^{+}\times \phi_{g'_{-}}^{-}$ on $C$. In particular, the value on the boundary remains unchanged and $\Psi^{-1}(\Delta)$ is still a compact set. By Proposition \ref{invcob}, the intersection number of the maps
\[
	\phi^{+}_{g'_{+}}: I(S,g'_{+})^{+} \rightarrow \dT \ \ \text{and} \ \  \phi_{g_{+}}^{+}: I(S,g_{+})^{+} \rightarrow \dT
\]
with the map
\[
	\phi_{g'_{-}}^{-}: I(S,g_{-})^{-} \rightarrow \dT \ ,
\] 
as defined in Section \ref{intersectiontheory}, must be the same. This gives a contradiction, because 
\[
	0=\Im(\phi^{+}_{g_{+}}, \phi_{g'_{-}}^{-})\ne \Im(\phi^{+}_{g'_{+}}, \phi_{g'_{-}}^{-})=1 \ .
\]
\indent So we have proved that $\phi_{g_{+}}^{+}(I(S,g_{+})^{+}) \cap \phi^{-}_{g'_{-}}(I(S,g'_{-})^{-}) \ne \emptyset$, but we do not know if the intersection is transverse. Repeating the above argument choosing the closed set $C=(X\setminus B)\cup I(S,g'_{+})^{+}$, we obtain that a perturbation $\psi$ of $\phi_{g_{+}}^{+}\times \phi^{-}_{g'_{-}}$ which is transverse to $\Delta$ and coincides with $\phi_{g_{+}}^{+}\times \phi^{-}_{g'_{-}}$ outside the interior of a compact set containing $(\phi_{g_{+}}^{+}\times \phi^{-}_{g'_{-}})^{-1}(\Delta)$ has intersection number $\Im(\psi,\Delta)=1$. By Proposition \ref{invpert}, every perturbation of the map $\phi_{g_{+}}^{+}\times \phi^{-}_{g'_{-}}$ obtained in this way has intersection number with $\Delta$ equal to $1$. \\
\\
\indent This enables us to deform the metric $g_{-}$ without losing the intersection, by repeating a similar argument. Suppose by contradiction that 
\[
	\phi_{g_{+}}^{+}(I(S,g_{+})^{+}) \cap \phi^{-}_{g_{-}}(I(S,g_{-})^{-}) = \emptyset \ .
\]
Consider the manifold
\[
	W^{-}=\bigcup_{t \in [0,1]} I(S, g^{-}_{t})^{-}
\]
and the map 
\[
	\Phi^{-}\times \phi_{g_{+}}^{+}: X=W^{-} \times I(S, g_{+})^{+} \rightarrow (\T(S))^{4}=Y \ .
\]
By assumption the restriction of the map $\Phi^{-}\times \phi_{g_{+}}^{+}$ to the first boundary component $X_{0}=I(S,g_{+})^{+}\times I(S,g_{-})^{-}$ is transverse to $\Delta$ and by Proposition \ref{compact}, the pre-image $D=(\Phi^{-}\times  \phi_{g_{+}}^{+})^{-1}(\Delta)$ is a compact set. Let $B$ be the interior of a compact set containing $D$ and let $C=(X \setminus B)\cup X_{0}$. By construction, $\Phi^{-}\times \phi_{g_{+}}^{+}$ is transverse to $\Delta$ along the closed set $C$. Applying Theorem \ref{perturb}, there exists a smooth map $\Psi: X \rightarrow Y$ which is transverse to $\Delta$ and which coincides with $\Phi^{-}\times \phi_{g_{+}}^{+}$ on $C$. In particular, the value on the boundary $X_{0}$ remains unchanged and $\Psi^{-1}(\Delta)$ is still a compact set. Moreover, the value of $\Psi$ on the other boundary component is a perturbation of $\phi_{g_{+}}^{+}\times \phi^{-}_{g'_{-}}$ which is transverse to $\Delta$ and coincides with $\phi_{g_{+}}^{+}\times \phi^{-}_{g'_{-}}$ outside the interior of a compact set containing $(\phi_{g_{+}}^{+}\times \phi^{-}_{g'_{-}})^{-1}(\Delta)$. Hence, by Proposition \ref{invpert} and by Proposition \ref{invcob}, the intersection number $\Im(\phi^{+}_{g_{+}}, \phi_{g_{-}}^{-})$ must be equal to $1$, thus giving a contradiction.

\begin{oss}\label{unique?} It might be possible to prove the uniqueness of this intersection by applying Proposition \ref{unico}. To this aim, it would be necessary to show that for every couple of metrics $g_{+}$ and $g_{-}$ with curvature strictly smaller than $-1$, the functions	$\phi_{g_{+}}^{+}: I(S,g_{+})^{+} \rightarrow \dT$ and $\phi_{g_{-}}^{-}: I(S,g_{-})^{-} \rightarrow \dT$ are transverse. 
\end{oss}

\section*{Acknowledgment}
 This is a part of a PhD project I have been doing under the supervision of Prof. Jean-Marc Schlenker. I would like to thank him for his patient guidance and invaluable support throughout this work. I am also grateful to Nicolas Tholozan and Daniel Monclair for our fruitful conversations about the subject.

\bigskip

\noindent \footnotesize \textsc{DEPARTMENT OF MATHEMATICS, UNIVERSITY OF LUXEMBOURG}\\
\emph{E-mail address:}  \verb|andrea.tamburelli@uni.lu|

\end{document}